\documentclass[12pt]{article}

\usepackage{amsfonts}
\usepackage{xr}
\usepackage{graphicx,psfrag}
\usepackage{amsmath,amsthm}
\usepackage{color,soul}
\usepackage[active]{srcltx}

\def\2{C^{1,2}(\R\times\R^N)}
\def\to{\rightarrow}
\def\e{\varepsilon}

\def\O{\Omega}

\def\R{\mathbb{R}}

\def\tilde{\widetilde}
\def\.{\cdot}

\DeclareMathOperator*{\essinf}{ess{\,}inf}

\def\lm#1{\left\lfloor{#1}\right\rfloor}

\newcommand{\be}{\begin{equation}}

\newcommand{\ee}{\end{equation}}
\newcommand{\baa}{\begin{array}}
\newcommand{\eaa}{\end{array}}
\newcommand{\ba}{\begin{eqnarray}}
\newcommand{\ea}{\end{eqnarray}}

\newtheorem{thm}{\bf Theorem}[section]
\newtheorem{lem}[thm]{\bf Lemma}
\newtheorem{prop}[thm]{\bf Proposition}
\newtheorem{cor}[thm]{\bf Corollary}
\theoremstyle{definition}

\newenvironment{formula}[1]{\begin{equation}\label{eq:#1}}
                       {\end{equation}\noindent}
\def\Fi#1{\begin{formula}{#1}}
\def\Ff{\end{formula}\noindent}

\begin{document}
\date{}

\title{On the instability of threshold solutions of reaction-diffusion equations, \\ and applications to  optimization problems.}

\author{Gr\'egoire Nadin\footnote{Institut Denis Poisson, Universit\'e d'Orl\'eans, Universit\'e de Tours, CNRS, Orl\'eans, France ({\tt gregoire.nadin@univ-orleans.fr}).}}
\maketitle


\relpenalty=10000
\binoppenalty=10000


\begin{abstract}
The first part of this paper is devoted to the derivation of a technical result, related to the stability of the solution of a reaction-diffusion equation $u_t-\Delta u = f(x,u)$ on $(0,\infty)\times \R^N$, where the initial datum $u(0,x)=u_0(x)$ is such that $\lim_{t\to +\infty} u(t,x)=W(x)$ for all $x$, with $W$ a steady state in $H^1(\R^N)$. We characterize the perturbations $h$ such that, if $u^h$ is the solution associated with the initial datum $u_0+h$, then, if $h$ is small enough in a sense, one has $u^h(t,x)>W(x)$ (resp. $u(t,x)<W(x)$) for $t$ large. This condition depends on the sign of $\int_{\R^N} h(x)p(0,x)dx$, where $p$ is an adjoint solution, which satisfies a backward parabolic equation on $(0,\infty)$ and is uniquely defined \cite{Huskaunbounded}. 

We then provide two applications of our result. We first address an open problem stated in \cite{Garnier} when $N=1$ and $f$ is a bistable nonlinearity independent of $x$. Namely, we compute the derivative of the critical length $L^*(r)$ associated with the initial datum $\mathbf{I}_{(-L-r,-r)\cup (r,L+r)}$, that is the length $L$ above (resp. below) which $u(t,x)$ converges to $1$ (resp. $0$) as $t\to +\infty$. 

Lastly, again when  $N=1$ and $f$ is a bistable nonlinearity independent of $x$, we prove the existence and characterize with a bathtub principle the initial datum $\underline{u}_0$ minimizing some cost function $\int_\R j(u_0)$ and guaranteeing at the same time that $\liminf_{t\to +\infty} u(t,x)>0$ for all $x$. 
\end{abstract}

\maketitle

\smallskip



\section{Introduction and statement of the main result}

\subsection{Statement of the problem}

We address in this paper the following optimization problem. Consider the solution $u$ of the reaction-diffusion equation
\begin{equation} \label{eq:princ} u_t- \Delta u =f(u) \quad \hbox{ on } (0,\infty) \times \R^N, \quad u(0,x)=u_0(x) \hbox{ for all  } x\in \R^N\end{equation}
associated with an initial datum $u_0$. We will mostly consider in this paper bistable reaction-diffusion nonlinearities, that is, $f\in \mathcal{C}^2([0,1]$ satisfying
\begin{equation} \label{hyp:fbist1}\tag{Hbist}\begin{array}{l}
\exists \beta^*\in (0,1) \hbox{ s.t. } \int_0^{\beta^*}f=0, \ \int_0^{u}f<0,  \ \forall u\in (0,\beta^*), \\ 
\\
f(u)>0, \ \forall u\in (\beta^*,1), \  \\ 
\\
f'(1)<0, \ f'(0)<0.\\ \end{array}\end{equation}
Under such hypotheses, we can define $\theta:= \sup \{ s\in (0,1), \ f(s)=0\}$.

We want to determine what are the initial data $u_0$ ensuring a convergence to the steady state $1$, and to minimize, in a sense, $u_0$ under this constraint. More precisely, consider a closed subset $\mathcal{A}\subset \{ u_0\in L^\infty (\R^N), \quad 0\leq u_0\leq 1\}$ as a class of admissible initial data, and $J:\mathcal{A}\to \R^+$ a cost function. We want to characterize the functions 
\begin{equation} \label{eq:optpbm}\underline{u}_0\in \mathcal{A} \hbox{ such that } J(\underline{u}_0)=\inf\big \{ J(u_0), u_0\in \mathcal{A}, u(t,x)\to 1 \hbox{ loc. unif. in } x \hbox{ as } t\to +\infty\big\}.\end{equation}

In order to illustrate the main features of this problem, consider the trivial case where $\mathcal{A}$ is the set of constant functions with respect to $x$, and $J(u_0)=u_0$. Clearly, the solution $u$ of (\ref{eq:princ}) does not depend on $x$ in this case, it satisfies $u'=f(u)$ in $(0,\infty)$, $u(0)=u_0$. Hence,  as $t\to +\infty$, it converges to $1$ if $u_0>\theta$, it stays constant equal to $\theta$ if $u_0=\theta$, and it stays below $\theta$ if $u_0<\theta$. We thus want to minimize $J(u_0)=u_0$ over the constant functions $u_0>\theta$. The infimum is reached for $\underline{u}_0=\theta$, which does not lead to a convergence to $1$ but to the unstable steady state $\theta$, hence it is not a minimum. Moreover, if $u_0>\underline{u}_0$, then $u(t)\to 1$ as $t\to +\infty$, while  if $u_0\leq \underline{u}_0$, then $\limsup_{t\to +\infty}u(t)\leq \theta$. This trivial example shows that initial data leading to a convergence to unstable steady states will play a crucial role in the optimization problem (\ref{eq:optpbm}). 

We will first prove a general technical result for a multi-dimensional heterogeneous equation in Section \ref{sec:tech}, characterizing more precisely the stability of an initial datum converging to a threshold solution $W\in H^1 (\R^N)$. We will then present two applications of this result in Sections \ref{sec:app1} and \ref{sec:app2}.


\subsection{Statement of the technical result}\label{sec:tech}

We consider in this section a general parabolic equation over $\R^N$, with a nonlinearity depending on $x$:

\begin{equation} \label{eq:hetero} u_t- \Delta u =f(x,u) \quad \hbox{ on } (0,\infty) \times \R^N, \quad u(0,x)=u_0(x) \hbox{ for all  } x\in \R^N\end{equation}

Here, $f$ does not necessarily satisfies (\ref{hyp:fbist1}), and we only assume that $f=f(x,u)$ is $\mathcal{C}^2$ with respect to $u \in [0,1]$ uniformly in $x\in \R^N$ and $L^\infty (\R^N\times [0,1])$, and that 
\begin{equation} \label{hyp:f} \tag{H0} \exists \sigma >0 \hbox{ s.t. } \quad \forall x\in \R^N, \quad  f(x,0)=f(x,1)=0, \quad  \partial_u f(x,0)\leq -\sigma. \end{equation}

We make the following hypothesis 
\begin{equation}  \label{eq:groundstate}\tag{H1} \hbox{there exists an unstable positive solution } W\in H^1 (\R^N) \hbox{ of }
 -\Delta W = f(x,W) \hbox{ in } \R^N.\end{equation}
 
 In the sequel, we will use the terminology introduced by Polacik (see \cite{Polacik} for example) and  call such a solution a threshold solution.
 Here, by unstable, we mean the following. Consider the principal eigenvalue 
$\lambda \in \R$ associated with the linearization around the steady state $W$, 
that is, the unique $\lambda \in \R$ such that there exists 
a solution $\varphi\in H^1 (\R)$ of 
\begin{equation} \label{eq:eigen}-\Delta\varphi- \partial_u f(x,W(x))\varphi = \lambda \varphi \quad \hbox{ in } \R^N,\quad \varphi >0.\end{equation}
When we say that $W$ is unstable, we mean that $\lambda<0$. 

The existence and uniqueness of $(\lambda,\varphi)$ are classical since $x\mapsto \partial_u f(x,W(x))$ is negative for $|x|$ large (here $\varphi$ is unique up to multiplication by a positive constant). 
In the sequel we will always consider the normalization $\|\varphi\|_{L^\infty(\R^N)}=1$.

Note that as $W\in H^1 (\R^N)$, one could easily prove that for all $\delta<\sqrt{\sigma}$, there exists $A>0$ such that $W(x)\leq A e^{-\delta |x|}$. Also, it is well-known that hypothesis (\ref{eq:groundstate}) is satisfied if $f$ does not depend on $x$ and the bistability hypothesis (\ref{hyp:fbist1}) is satisfied.

The next Lemma is classical and proved in Section \ref{sec:prel} for the sake of completeness. 
\begin{lem}\label{lem:lambda<0} If $f$ does not depend on $x$, then for any positive solution $W\in H^1 (\R^N)$ of 
$ -\Delta W = f(W) \hbox{ in } \R^N$, one has $\lambda<0$. 
\end{lem}

We will now adapt the classical control and optimization theory to the present framework by introducing an adjoint solution $p$ associated with (\ref{eq:hetero}). As the problem is stated for all positive time $t\in (0,\infty)$, $p$ needs to be defined for all $t\in (0,\infty)$, and to be the solution of the adjoint equation associated with the linearization of (\ref{eq:hetero}) near $u$. In other words, it will be the solution of a backward parabolic equation over $(0,\infty)$. Such solutions are not always defined. However, here we will consider $u_0$ a measurable initial datum such that $0\leq u_0\leq 1$ in $\R^N$ and  $u(t,x)\to W(x)$ uniformly in $x$ as $t\to +\infty$, with $W\in H^1 (\R^N)$, where $u$ is the solution of the Cauchy problem (\ref{eq:hetero}). Using this hypothesis, we  will be able to prove in 
Proposition \ref{prop:defp}, using a strong result due to Huska and Polacik in \cite{Huskaunbounded}, that the following equation admits a unique positive solution $p$ up to normalization:
$$-p_t-\Delta p = f'\big(x, u(t,x)\big)p \hbox{ in } (0,\infty)\times \R^N.$$

We are now in position to state our main result. 
We say that a perturbation $h\in L^\infty (\R^N)$ is admissible if $0\leq u_0+ h\leq 1$. %
%

\begin{thm}\label{thm:h} Assume (\ref{hyp:f}) and (\ref{eq:groundstate}) and consider 
 $u_0$ a measurable initial datum such that $0\leq u_0\leq 1$ in $\R^N$ and  $u(t,x)\to W(x)$ uniformly in $x$ as $t\to +\infty$, where $u$ is the solution of the Cauchy problem (\ref{eq:hetero}). 
 
Then, for all $c>0$ and $q\in [1,\infty]$, there exists $\e_q>0$ such that for all 
admissible perturbation $h\in L^q (\R^N)$, if $u^h$ is the solution associated with the initial datum $u_0+h$, then 
\begin{itemize}
\item if $\int_{\R^N} p(0,x)h(x)dx \geq c\|h\|_{L^q(\R^N)}$ and $\|h\|_{L^q(\R^N)}\leq \e_q$, one has, $ \liminf_{t\to +\infty} u^h(t,x)>W(x)$ for all $x\in \R^N$, 
\item if $\int_{\R^N} p(0,x)h(x)dx \leq -c\|h\|_{L^q(\R^N)}$ and $\|h\|_{L^q(\R^N)}\leq \e_q$, one has, $ \limsup_{t\to +\infty} u^h(t,x)<W(x)$ for all $x\in \R^N$,
\end{itemize}
where $p=p(t,x)$ is uniquely defined by Proposition \ref{prop:defp}.
\end{thm}

We provide several examples in the next sections of initial data satisfying the hypotheses of Theorem \ref{thm:h}, together with applications of this result. 

As a corollary of Theorem \ref{thm:h}, we derive the following result, which identifies somehow the stable manifold near $u_0$ associated with the threshold solution $W$. 

\begin{cor}\label{cor:h}
Under the hypotheses of Theorem \ref{thm:h}, 
one has $\int_{\R^N} p(t,x)\partial_t u (t,x)dx=0$ for almost every $t>0$. 
\end{cor}

Note that it is not clear whether $\int_{\R^N} p(t,x)\partial_t u (t,x)dx$ makes sense when $t=0$ or not since $u_0$ could be very singular (only $L^\infty$). This is why we need to state this result for almost every $t>0$. 

Lastly, let us state a simplified version of Theorem \ref{thm:h}, by doing two simplifying hypotheses. First, let us consider $f$ independent of $x$ satisfying the bistability hypothesis (\ref{hyp:fbist1}). Second, assume that the perturbation takes the form $u_0+\e h$ for some given $h$.

\begin{cor}\label{cor:h} Assume that $f$ does not depend on $x$ and satisfies (\ref{hyp:fbist1}). Consider 
 $u_0$ a measurable initial datum such that $0\leq u_0\leq 1$ in $\R^N$ and  $u(t,x)\to W(x)$ uniformly in $x$ as $t\to +\infty$, where $W\in H^1 (\R^N)$ is a solution of 
$ -\Delta W = f(W)$ in $\R^N$, and $u$ is the solution of the Cauchy problem (\ref{eq:hetero}). 

Then, for all 
admissible perturbation $h\in L^q (\R^N)$, there exists $\e_0=\e_0 (h)>0$ such that for all $\e\in (0,\e_0)$, if $u^{\e h}$ is the solution associated with the initial datum $u_0+\e h$, then 
\begin{itemize}
\item if $\int_\R p(0,x)h(x)dx >0$, one has $ \lim_{t\to +\infty} u^{\e h}(t,x)=1$ locally in $x\in \R^N$, 
\item if $\int_\R p(0,x)h(x)dx <0$, one has $ \lim_{t\to +\infty} u^{\e h}(t,x)=0$ uniformly in $x\in \R^N$,
\end{itemize}
where $p=p(t,x)$ is uniquely defined by Proposition \ref{prop:defp}.
\end{cor}

We leave the proof of this corollary to the reader since it is a straightforward application of Theorem \ref{thm:h}.


\subsection{Application 1}\label{sec:app1}

Assume that $f$ does not depend on $x$ and satisfies the additional bistability hypothesis (\ref{hyp:fbist1}).

Consider a family of initial data $(\psi_L)_{L>0}$ such that 
\begin{itemize}
\item $\psi_0\equiv 0$, for all $L>0$, $\psi_L$ is a nonnegative measurable initial datum such that $0\leq \psi_L\leq 1$,  
\item $L \in (0,\infty)\mapsto \psi_L \in L^1 (\R^N)$ is a continuous mapping, and monotone increasing in the sense that if $\ell<L$, then $\psi_\ell\leq \psi_L$ and the strict inequality holds on a set of positive measure,
\item there exists a ball $B\subset \R^N$ such that $\int_B \psi_L\to +\infty$ as $L\to +\infty$. 
\end{itemize}
Under these hypotheses, it has been proved in \cite{Polacik} that there exists $L^*>0$ such that if we denote by $u_L$ the solution associated with the initial datum $\psi_L$, one has 
\begin{itemize}
\item for all $0\leq L <L^*$, $u_L(t,x)\to 0$ as $t\to +\infty$ uniformly in $x$,
\item for all $ L >L^*$, $u_L(t,x)\to 1$ as $t\to +\infty$ locally in $x$,
\item for $ L =L^*$, $u_{L^*}(t,x)\to W(x)$ as $t\to +\infty$ uniformly in $x$, for some positive, radially symmetric with respect to some $\xi \in \R^N$, solution $W\in H^1 (\R^N)$ of (\ref{eq:groundstate}). 
\end{itemize}
The last statement could be found in Section 6 of \cite{Polacik}. 
This situation provides a family of examples of initial data giving rise to solutions converging to $W$ as $t\to +\infty$, hence satisfying the hypotheses of Theorem \ref{thm:h}. 

This type of properties has first been derived by Zlatos \cite{Zlatos} when $\psi_L:=\mathbf{I}_{(-L,L)}$, where $\mathbf{I}_{A}$ is the indicatrix function of the set $A$. It was known since the pioneering work of Kanel \cite{Kanel} that $u_L(t,x)\to 0$ as $t\to +\infty$ if $L$ is too small, while $u_L(t,x)\to 1$ as $t\to +\infty$ when $L$ is large enough. Zlatos proved that this transition is indeed sharp. There have been many generalizations of this result, we just cite here Polacik's \cite{Polacik}, that fits to our present framework. 

\bigskip

Our aim is to understand how the threshold $L^*$ depends on the parameters of the equation. 
Let us consider for $N=1$ and for all $r\geq 0$ the family $(\psi_L^r)_{L>0}$ defined by $\psi_L^r:= \mathbf{I}_{(-L-r,-r)\cup (r,L+r)}$. For all $r\geq 0$, this family satisfies the previous hypotheses and one could thus derive the existence of a critical length $L^*(r)$. We want to find $\underline{r}$ minimizing $r\geq 0 \mapsto L^*(r)$. 

This problem has been addressed by Garnier, Hamel and Roques \cite{Garnier}, who computed the limit of $L^*(r)$ when $r\to +\infty$, and performed  numerical simulations supporting the conjecture that $\underline{r}>0$ in some situations. In other words, there might exist $r>0$ and $L>0$ such that the solution associated with $\mathbf{I}_{(-L-r,-r)\cup (r,L+r)}$ converges to $1$, while the one associated with $\mathbf{I}_{(-L,L)}$ converges to $0$ as $t\to +\infty$. This conjecture has remained open since then. 

Here, we do not prove this conjecture, but we provide a computation of the derivative of the threshold $L^*(r)$ with respect to $r$.

\begin{prop} \label{prop:app1}The function $L^*:r\in (0,\infty)\to (0,\infty)$ is derivable and 
$$(L^*)'(r)=\frac{p^r(0,r)}{p^r(0,L^*(r)+r)}-1,$$
where $p^r$ is the unique (up to multiplication by a constant) positive solution of 
$$-p_t-\Delta p = f'\big( u_{L^*(r)}(t,x)\big)p \hbox{ in } (0,\infty)\times \R^N.$$
\end{prop}

As a corollary, if one manages to prove that $p^0(0,0)<p^0 (0,L^* (0))$, then $(L^*)'(0)<0$ and thus $\underline{r}>0$.  

The method we develop to prove this result is quite general and can enable us to address much more general families of functions  $(\psi^r_L)_{L>0}$ indexed by some $r>0$. Here, we just stick to $N=1$ and  $\psi_L^r:= 1_{(-L-r,-r)\cup (r,L+r)}$ since the computation of $(L^*)'(r)$ is already quite technical. We leave other types of applications for future works. 

Let us also mention here some related works. In \cite{ADF}, the authors computed an equivalent for $L^*(\e)$, where the nonlinearity reads $f(u)=u(1-u)(u+1/2+\e)$. In \cite{ADK} estimates are derived for a nonlocal equation. In \cite{FK}, another nonlocal equation arising in neural fields modeling is considered, for which some explicit computations could be carried out. In \cite{fragmentation}, the authors investigate the influence of fragmentation of the initial datum on the large time behaviour.


\subsection{Application 2}\label{sec:app2}

Consider the case $N=1$. Under hypotheses (\ref{hyp:f}), it has been proved by Du and Matano in \cite{DuMatano} that for any compactly supported initial datum $0\leq u_0\leq 1$, the associated solution $u$ converges to a stationary solution of (\ref{eq:princ}) which is either a constant or a symmetric decreasing solution with respect to some $\xi \in \R$. 

We could thus try to characterize the initial data $\underline{u}_0$ minimizing some cost function $\int_\R j(u_0)$ and, at the same time, giving rise to a solution of the Cauchy problem $u=u(t,x)$ taking off from $0$ as $t\to +\infty$. 

A related problem has been investigated by the author, with Mazari and Toledo, in \cite{MNT, NT}. Namely, for $T>0$ and $m  \in (0,|\O|)$ given, we investigated the maximization of $\int_\O u(T,x)dx$ with respect to initial data $u_0$ such that $0\leq u_0\leq 1$ and $\int_\O u_0=m$. In \cite{NT}, we proved that any maximizing initial datum $\overline{u}_0$ could be characterized thanks to some adjoint problem, and we provided some numerical simulations showing that the maximizers could be very singular. This result has been extended by Abdul Halim and El Smaily in \cite{ElSmaily} to heterogeneous questions, with an advection term. In \cite{MNT}, we pushed further the characterization by analyzing the abnormal set, that is, the set where the adjoint is constant. 

Here, we change our point of view: we want to minimize a cost function $\int_\R j(u_0)$, the constraint on $u_0$ being that the associated solution does not converge to $0$. We expect the problem addressed in \cite{MNT, NT} to be a good approximation of the present problem when $T$ is large enough, but we leave such an approximation result for a possible future work.

Let us first prove that the problem under scrutiny here admits a minimizer.  

\begin{lem}\label{lem:app1}Assume  (\ref{hyp:f}) and 
let 
$$\mathcal{B}:= \{ u_0\in L^1(\R), \ 0\leq u_0\leq 1, \quad \limsup_{t\to+\infty} \sup_{x\in \R}u(t,x)>0 \}.$$
Assume that $\mathcal{B}$ is not empty. Consider a convex cost function $j: [0,\infty)\to  [0,\infty)$ such that $j(0)=0$ and $j(s)>0$ for all $s>0$. Then 
there exists $\underline{u}_0\in \mathcal{B}$ minimizing $\int_{\R} j(u_0)$ over $\mathcal{B}$.
\end{lem}

We now assume that $f$ is bistable, in the sense of hypothesis (\ref{hyp:fbist1}). We would like to push further the characterization of $\underline{u}_0$. 

Under hypothesis (\ref{hyp:fbist1}), Matano and Polacik proved  (see Theorem 2.5 and its discussion in \cite{MatanoPolacik}), that for any continuous function $u_0$ such that $0\leq u_0\leq 1$ and $\lim_{x\to \pm \infty} u_0 (x)=0$, then $u(t,\cdot)$ converges as $t\to +\infty$, either locally uniformly to $1$,  uniformly to $0$, or uniformly to a positive threshold solution $W\in H^1 (\R)$ satisfying (\ref{eq:groundstate}). It is easy to see that this result still holds for any $u_0\in L^1 (\R)$ such that $0\leq u_0\leq 1$ by applying Matano-Polacik's result to $u(\tau,\cdot)$ for some $\tau>0$. Using this intermediate result, we will be able to show that $\underline{u}_0$ gives rise to a solution of the Cauchy problem $\underline{u}=\underline{u}(t,x)$ converging to some intermediate solution $W\in H^1 (\R)$. Hence, we could apply Theorem \ref{thm:h} to provide "bathtub characterization" of $\underline{u}_0$.

\begin{prop} \label{prop:bathtub} Assume (\ref{hyp:fbist1}). Consider a convex and derivable cost function $j: [0,\infty)\to  [0,\infty)$ such that $j(0)=0$ and $j(s)>0$ for all $s>0$. Let $\underline{u}=\underline{u}(t,x)$ associated with an initial datum $\underline{u}_0\in \mathcal{B}$ minimizing $\int_{\R} j(u_0)$ over $\mathcal{B}$. 

\begin{enumerate}
\item One has $\lim_{t\to +\infty} u(t,x)=W(x)$ uniformly in $x\in \R$, for some $W$ satisfying (\ref{eq:groundstate}). 
\item
Moreover, let $p$ the unique (up to multiplication by a constant) positive solution of 
$$-p_t-\Delta p = f'\big( \underline{u}(t,x)\big)p \hbox{ in } (0,\infty)\times \R.$$
There exists $c>0$ such that 
\begin{itemize}
\item for a.e. $x\in \R$ such that $p(0,x)>c j'(\underline{u}_0(x))$, one has $\underline{u}_0(x)=1$,
\item for a.e. $x\in \R$ such that  $p(0,x)<c  j'(\underline{u}_0(x))$, one has $\underline{u}_0(x)=0$.
\end{itemize}
It follows that $p(0,\cdot)\equiv c  j'(\underline{u}_0)$ almost everywhere on $\{0<\underline{u}_0<1\}$. 
\end{enumerate}

\end{prop}

It is also possible to state such a result in $\R^N$, by restricting $\mathcal{B}$ to radial initial data, using Polacik's results from \cite{Polacikradial}. We leave such a result for possible future works. 

Lastly, we derive an application of this result when $j(u)\equiv u$.

\begin{cor} \label{cor:app1} Assume (\ref{hyp:fbist1}). Then 
the function $W=W(x)$ is not a minimizer of $u_0\mapsto \int_\R u_0$ over $\mathcal{B}$. 
\end{cor}


\section{Proof of Theorem \ref{thm:h}}

\subsection{Preliminaries: principal Floquet bundles for linear parabolic equations}\label{sec:prel}

We describe and apply in this section the results proved by Huska and Polacik in \cite{Huskaunbounded}, enabling us to define uniquely the adjoint solution $p$ in particular.

\begin{lem}\label{lem:unifux}
Assume that $u_0$ is a measurable initial datum such that $0\leq u_0\leq 1$ in $\R^N$ and  $u(t,x)\to W(x)$ uniformly in $x$ as $t\to +\infty$. 
There exists $R>0$ such that $f'\big(x,u(t,x)\big)\leq 0$ for all $t\geq 0$ and $|x|\geq R$. 
\end{lem}

\begin{proof}
As $W\in H^1 (\R^N)$, one has $\lim_{|x|\to +\infty} W(x)=0$. Let $\kappa>0$ such that $f(x,s)\leq 0$ for all $x\in \R^N$, $s\in [0,\kappa]$, $R>0$ such that $W(x)<\kappa/2$ for all $|x|\geq R$, and $T>0$ such that $\|u(t,\cdot)-W\|_{L^\infty (\R^N)}<\kappa/2$ for all $t\geq T$.  It follows that $0\leq u(t,x)\leq \kappa$, and thus $f'\big(x,u(t,x)\big)\leq 0$, for all $t\geq T$ and $|x|\geq R$.

Next, we remark that 
$$u_t-\Delta u \leq Ku \quad \hbox{ in } (0,\infty)\times \R^N$$
with $K=\|\partial_u f\|_\infty$. 
Hence,
$$u(t,x)\leq \frac{e^{K t}}{(4\pi t)^{N/2}}\int_{\R^N} e^{-\frac{|x-y|^2}{4t}} u_0 (y) dy.$$
As $\lim_{r\to +\infty} \mathrm{esssup}_{|x|>r} u_0(x)=0$, one gets from the dominated convergence theorem that $u(t,x)\to 0$ as $|x|\to +\infty$ locally uniformly with respect to $t\geq 0$. Hence, even if it means increasing $R$, we can assume that $u(t,x)\leq \kappa$ for all $t\geq 0$ and $|x|\geq R$. 
\end{proof}

Let us extend $u$ on $\R$ by $u(t,x):=u_0 (x)$ for all $t\leq 0$.  
The following result follow easily from \cite{Huskaunbounded}. 

\begin{prop}\label{prop:defp}
Assume that $u_0$ is a measurable initial datum such that $0\leq u_0\leq 1$ in $\R^N$ and  $u(t,x)\to W(x)$ uniformly in $x$ as $t\to +\infty$. 
There exist two positive solution $p$ and $v$ of 
$$-p_t-\Delta p = f'\big(x, u(t,x)\big)p \hbox{ in } \R\times \R^N,$$
$$v_t-\Delta v = f'\big(x, u(t,x)\big)v \hbox{ in } \R\times \R^N.$$
These positive solutions are unique up to multiplication by a positive constant.

Moreover, for all $0<\delta<\sqrt{-\lambda}$, there exists $C>0$ such that
$$p(t,x)\leq C e^{-\delta x}\|p(t,\cdot)\|_{L^\infty(\R)},$$
$$v(t,x)\leq C e^{-\delta x}\|v(t,\cdot)\|_{L^\infty(\R)},$$
and
$$\|p(t,\cdot)\|_{L^{\infty}(\R^N)}\leq Ce^{-\delta^2 (t-s)}\|p(s,\cdot)\|_{L^{\infty}(\R^N)} \quad \hbox{ for all } t>s.$$
$$\|v(t,\cdot)\|_{L^{\infty}(\R^N)}\geq \frac{1}{C}  e^{\delta^2 (t-s)}\|v(s,\cdot)\|_{L^{\infty}(\R^N)} \quad \hbox{ for all } t>s.$$
In the sequel, we will normalize these functions by $\|v(0,x)\|_{L^{\infty}(\R^{N})}=1$ and $\int_{\R^{N}} p(0,x)v(0,x)dx=1$. 
\end{prop}

\begin{proof}
This follows from Theorem 2.1 of \cite{Huskaunbounded} 
as soon as we can verify hypothesis (H2) of  \cite{Huskaunbounded}. That is, we need to prove that for all $\tilde{\lambda}>\lambda$, there exists $C>0$ such that for all $s_0\in \R$, there exists a solution $\tilde{v}$ of 
$$\partial_t \tilde{v}- \Delta \tilde{v} = f'\big( x,u(t,x)\big)\tilde{v} \quad \hbox{ in } (s_0,\infty)\times \R^N$$
such that 
$$\|\tilde{v}(t,\cdot)\|_{L^{\infty}(\R^N)}\geq e^{(\tilde{\lambda}-\lambda)(t-s)}\|\tilde{v}(s,\cdot)\|_{L^{\infty}(\R^N)} \quad \hbox{ for all } t>s\geq s_0.$$
We take $\tilde{v}(s_0,x)\equiv 1$. Let $T>0$ so that:
$$|f'\big( x,u(t,x)\big)-f'\big( x,W(x)\big)|\leq\tilde{\lambda}-\lambda \quad \hbox{ for all } x\in \R^N \hbox{ and } t\geq T.$$
If $s_0<T$, then one can easily prove using the comparison principle that $\tilde{v}(t,x)\geq e^{-K(t-s_0)}$ for all $t>s_0$ for $K=\|f'\|_\infty$. In particular, $\tilde{v}(T,x)\geq e^{-K(T-s_0)}\geq e^{-K(T-s_0)}\varphi (x)$ since $\varphi\leq 1$. Moreover, one has:
$$\partial_t \tilde{v}- \Delta \tilde{v} \geq  \Big(f'\big( x,W(x)\big)-\delta^2\Big) \tilde{v} \quad \hbox{ in } (T,\infty)\times \R^N.$$
As $\varphi (x)e^{-K(T-s_0)-\tilde{\lambda}(t-T)}$ is a subsolution of this equation, we obtain
$$\tilde{v}(t,x)\geq e^{-K(T-s_0)-\tilde{\lambda}(t-T)}\varphi (x) \quad \hbox{ for all } t>T, w\in \R^N.$$
It follows that 
$$\|\tilde{v}(t,\cdot)\|_{L^{\infty}(\R^N)}\geq e^{-K(T-s_0)-\tilde{\lambda}(t-T)} \quad \hbox{ for all } t>T.$$
This proves Hypothesis (H2) of  \cite{Huskaunbounded} and the result follows. 
\end{proof}

For any $h\in L^1 (\R^N)$, one can define
\begin{equation} \label{eq:thetadot}\left\{\begin{array}{l}
\partial_t \dot{u}-\Delta\dot{u}-f'\big(x,u(t,x)\big)\dot{u}=0 \quad \hbox{ in } (0,\infty)\times \R^N, \\
\dot{u}(0,x)=h(x).\\
\end{array}\right. \end{equation}

The next result shows exponential separation between the solutions of (\ref{eq:thetadot}) that are orthogonal to $p$ and the principal Floquet bundle generated by $v$. 

\begin{thm}\label{thm:Huskaunbounded} [Theorem 2.2 of \cite{Huskaunbounded}.]
Assume that $u_0$ is a measurable initial datum such that $0\leq u_0\leq 1$ in $\R^N$ and  $u(t,x)\to W(x)$ uniformly in $x$ as $t\to +\infty$. 
There exists $\gamma>0$ and $C>0$ such that for all $h\in L^\infty (\R^N)$, if $\int_{\R^N} h(x)p(0,x)dx=0$, then for all $t>s>0$:
$$\frac{\|\dot{u}(t,\cdot)\|_{L^\infty(\R^N)}}{\|v(t,\cdot)\|_{L^\infty(\R^N)}}\leq C e^{-\gamma (t-s)} \frac{\|\dot{u}(s,\cdot)\|_{L^\infty(\R^N)}}{\|v(s,\cdot)\|_{L^\infty(\R^N)}}.$$
\end{thm}

The next lemma states that the principal Floquet bundle converges, in a sense, to the classical notion of principal eigenvalue when the linearized equation becomes independent of $t$. 
\begin{lem}\label{lem:cvphi}
One has 
$$p(t,\cdot)/\|p(t,\cdot)\|_{L^\infty(\R^N)}\to \varphi \quad \hbox{ and }\quad v(t,\cdot)/\|v(t,\cdot)\|_{L^\infty(\R^N)}\to \varphi  \quad \hbox{ as } t\to +\infty$$
uniformly in $x$, where $\varphi$ is the principal eigenfunction defined by (\ref{eq:eigen}) and normalized by $\|\varphi\|_{L^\infty (\R^N)}=1$.
\end{lem}

\begin{proof}
These are immediate consequences of the uniqueness result stated in Proposition 2.5 of \cite{Huskaunbounded}. 
\end{proof}

%

Lastly, the notion of principal Floquet bundle on $\R^N$ could be approximated by the one on $B_R$ when $R\to +\infty$, in the following sense.

\begin{prop}\label{propHuskaRgrand}
For all $R>0$, there exists a unique (up to multiplication by a positive constant) positive solution $v_R$ of
$$\left\{ \begin{array}{lll} 
\partial_t v_{R}-\Delta v_{R} = f'\big( x,u(t,x)\big)v_{R} &\hbox{ in }& \R\times B_R,\\
v_{R}=0&\hbox{ on  }& \R\times \partial B_R.\\
\end{array}\right. $$
Moreover, if we normalize it by $\| v_R\|_{L^{\infty}(B_{R})}=1$, then one has 
$$\lim_{R\to +\infty}v_R(t,x)=v(t,x) \quad \hbox{ loc. unif. in } (t,x)\in \R\times \R^{N}.$$
Lastly, for all $\tilde{\lambda}\in (\lambda,0)$, there exists $R_0$ such that for all $R\geq R_0$:
$$\|v_{R}(t,\cdot)\|_{L^{\infty}(B_{R})}\geq e^{-\tilde{\lambda}(t-s)}\|v_{R}(s,\cdot)\|_{L^{\infty}(B_{R})} \quad \hbox{ for all } t>s.$$
\end{prop}

\begin{proof} 
This is proved along the proof of Theorems 2.1 and 2.4 in \cite{Huskaunbounded}.
\end{proof}

We conclude this section with the proof of Lemma \ref{lem:lambda<0}.

\begin{proof}[Proof of Lemma \ref{lem:lambda<0}.]
We notice that $\partial_{x_i}W$ is an $H^1(\R^N)$ eigenfunction associated with the eigenvalue $\tilde{\lambda}=0$. As the principal eigenvalue $\lambda$ is the smallest eigenvalue, one has $\lambda\leq \tilde{\lambda}=0$. Moreover, if $\lambda= \tilde{\lambda}=0$, then $\partial_{x_i}W$ is proportional to the principal eigenfunction, which as a sign. This would mean that $W$ is strictly monotone with respect to $x_i$ for all $i$. This would contradict $W\in H^1 (\R^N)$. Hence, $\lambda<0$. 
\end{proof}

\subsection{The case of a perturbation which is positive on a large ball}

We begin with the case where the perturbation $h$ is positive on a ball $B_R$, with $R$ large.

\begin{prop}\label{prop:R}
There exist $\e_\infty$ and $R_0>0$ such that for all  $R\geq R_0+1$, for any admissible perturbation $h\in L^\infty (\R)$ such that $\essinf_{B_R} h>0$ and $\|h\|_{L^\infty(\R^N)}<\e_\infty$, one has:
$$\inf_{(t,x)\in (0,\infty)\times B_{R-1}}(u^h-u)(t,x)>0.$$

\end{prop}


\begin{proof} Our aim is to construct a subsolution $\underline{u}$ of the equation satisfied by $u^h$, using the linearized equation near $u$.

Take $\tilde{\lambda}\in (\lambda,0)$, $R_0$ as in Proposition \ref{propHuskaRgrand}, $R> R_0+1$, $r\in (R-1,R)$ and $h$ satisfying the hypotheses of the Proposition. 

Define 
$$\beta(t):= \ln(\|v_{r} (t,\cdot)\|_{L^{\infty}(B_r)}).$$
Hence, if we define 
$$\lm{\beta'}:=	\sup_{t>0}\bigg(\inf_{s>0}\frac1t\int_s^{s+t}\beta'(\tau)d\tau\bigg),$$
one has $\lm{\beta'}> -\tilde{\lambda}$ by Proposition \ref{propHuskaRgrand}. 
Lemma 3.2 of \cite{NR} yields that 
$$\lm{\beta'}	=\sup_{A\in W^{1,\infty}(0,\infty)}\left(\essinf_{(0,\infty)}(\beta'+A')\right).$$
Take $A\in W^{1,\infty}(0,\infty)$
	such that $\beta'+A'\geq -\tilde{\lambda}$ a.e.~in $(0,\infty)$. 
Let 
$$w_r(t,x):=v_r(t,x)e^{-\beta(t)-A(t)}.$$
Even if it means adding a constant to $A$, we could assume that $w_r(0,\cdot)\equiv v_r(0,\cdot)$. 
	This function satisfies 
$$\partial_{t}w_{r}-\Delta w_{r} - \partial_u f\big( x,u(t,x)\big)w_{r}\leq \tilde{\lambda}w_{r} \hbox{ in } (0,\infty)\times B_r,$$
with $\sup_{(0,\infty)\times B_{r} }w_{r}\leq e^{\|A\|_\infty}$ and $\inf_{(0,\infty)\times B_{r'} }w_{r}>0$ for any $r'\in (0,r)$ thanks to the Harnack inequality. We could thus take $c$ small enough so that $v(0,x)\geq c w_{r}(0,x)$ for all $x\in B_{r}$, where $v$ is defined by Proposition \ref{prop:defp}.

Next, let us prove the following claim: there exists $\delta >0$ such that one has $\partial_n v_r (t,x)\leq -\delta \|v_r (t,\cdot)\|_{L^\infty(B_r)}$ for all $x\in \partial B_r$ and $t>0$. If this was not true, there would exist a sequence $\big((t_k,x_k)\big)_k$ in $(0,\infty)\times \partial B_r$ such that $\partial_n v_r (t_k,x_k)\geq -\frac{1}{k} \|v_r (t_k,\cdot)\|_{L^\infty(B_r)}$. If $(t_k)_k$ is bounded, we can assume up to extraction that $\big((t_k,x_k)\big)_k$ converges to a limit $(t_*,x_*)$ in $[0,\infty)\times \partial B_r$ such that 
$\partial_n v_r (t_*,x_*)\geq 0$, which contradicts the Hopf Lemma. If $t_k\to +\infty$ as $k\to +\infty$ along a subsequence, then as the sequence $\big( v_r(t_k,\cdot)/\|v_r(t_k,\cdot)\|_{L^\infty (B_r)}\big)_k$ is bounded in $L^\infty (B_r)$, using parabolic regularity, it converges to a function $v_*(0,\cdot)$ in $\mathcal{C}^2 (B_r)$ as $k\to +\infty$, that we could extend to a nonnegative solution $v_*$ of 
$$\partial_t v_* - \Delta v_* - \partial_u f\big( x,W(x)\big)v_*=0 \hbox{ in } \R\times B_r, \quad v_*=0 \hbox{ over } \R\times \partial B_r.$$
Moreover, one has $\partial_n v_* (0,x_*)\geq 0$. The Hopf Lemma yields that $v_* (0,\cdot)\equiv 0$ in $B_r$, a contradiction since $\|v_*\|_{L^\infty (B_r)}=1$. This proves the claim, from which it follows from the definition of $w_r$ that 
$\partial_n w_r \leq -\delta e^{-\|A\|_\infty}$ for all $x\in \partial B_r$ and $t>0$.

Even if it means increasing $r$, we can assume that $\partial_u f\big(x, u(t,x)\big)\leq \partial_u f(x,0)/2$ for all $t\geq 0$ and $|x|\geq r$ by Lemma \ref{lem:unifux} since $\partial_u f(x,0)<0$. 
Let $\alpha>0$ and define $w(x):= B\big(e^{-\alpha (|x|-r)}-1\big)$, for some small $B>0$ that  will be prescribed later. For $\alpha>0$ small enough (independent of $B$), one has 
$$-\Delta w\leq 0 \leq \Big(\partial_u f\big( x,u(t,x)\big)- \partial_u f(x,0)/2\Big) w \quad \hbox{ in } \R^N\backslash B_r.$$
since $w(x)<0$ in $ \R^N\backslash B_r$.

We now define $\underline{w}(x):= w_r(x)$ if $x\in B_r$, $w(x)$ if $|x|\geq r$. 
First, one has $\partial_n w(x)=-\alpha B$ for all $x\in \partial B_R$, where $n$ is the normal vector leaving $B_r$. We now take 
$B\leq \delta  e^{-\|A\|_\infty}/\alpha$. This yields
 $\partial_n w=-\alpha B\geq \partial_n w_r$ over $\partial B_r$ if $B$ is large enough. 
 Second,  as $w(x)<0$ and  $\partial_u f\big(x, u(t,x)\big)\leq \partial_u f(x,0)/2$  if $|x|\geq r$ for all $t>0$ and $\tilde{\lambda}<0$, 
one has 
\begin{equation} \label{eq:ubar2}\partial_{t}\underline{w}-\Delta \underline{w}- \partial_u f\big( x,u(t,x)\big)\underline{w}\leq -\sigma |\underline{w}| \hbox{ in } (0,\infty)\times \R^N,\end{equation}
in the weak sense since there is a jump of the derivatives over $\partial B_r$, where we choose $\sigma>0$ such that $\sigma<\min \{ -\sup_{x\in \R^N} \partial_u f(x,0)/2, -\tilde{\lambda}\}$. 

Let $\kappa_0>0$ such that for all $\kappa \in \R$ such that $|\kappa|\leq \kappa_{0}$, 
$$\Big|f(x,u(t,x)+\kappa)-f(x,u(t,x))-\partial_u f(x,u(t,x))\kappa \Big| < \sigma |\kappa |.$$
As $\sup_{(0,\infty)\times B_{r} }w_{r}<\infty$ and $\sup_{(0,\infty)\times (\R^N\backslash B_{r}) }|w|\leq B$, we can assume that $\kappa$ is small enough so that 
\begin{equation} \label{eq:ubar1}\Big|f\big(x,u(t,x)+\kappa \underline{w}(t,x)\big)-f\big(x,u(t,x)\big)-\partial_u f\big(x,u(t,x)\big)\kappa \underline{w}(t,x)\Big| < \sigma |\kappa \underline{w}(t,x)|\end{equation}
 for all $(t,x)\in (0,\infty)\times B_{r}$.

Define $\underline{u}(t,x):= u(t,x)+\kappa\underline{w}(t,x) $. Gathering (\ref{eq:ubar2}) and (\ref{eq:ubar1}), one gets
$$\partial_{t}\underline{u}-\Delta \underline{u}<f(x,\underline{u}) \quad \hbox{ on } (0,\infty)\times \R^N.$$

As  $\essinf_{B_R} h>0$ by hypothesis, we can take $\kappa$ small enough such that $h\geq \kappa w$ on $B_R$. Let $\e_\infty := \kappa B (1-e^{-\alpha (R-r)})$. Then for all $h$ such that $\|h\|_{L^\infty(\R^N)}<\e_\infty$, noticing that $w(x)\leq -B (1-e^{-\alpha (R-r)})$ for all $x$ such that $|x|\geq R$, one has $h\geq \kappa w$ 
over $\R^N\backslash B_R$. Hence $h\geq \kappa w$ on $\R^N$. 


As
$$\underline{u}(0,x)= u(0,x)+\kappa\underline{w}(0,x) \leq u(0,x)+ h(x)=u^h(0,x) \hbox{ for all } x\in \R^N,$$
the parabolic comparison principle yields
$$\underline{u}\leq u^h \quad \hbox{ on } (0,\infty)\times \R^N.$$
This implies:
 $$\inf_{B_{R-1}}\big(u^h(t,\cdot)-u(t,\cdot)\big)\geq \kappa\inf_{B_{R-1}}w_{r}>0  \quad \hbox{ on } (0,\infty),$$
 which concludes the proof. 

\end{proof}


\subsection{The general case}

We are now in position to prove our main result.

\begin{proof}[Proof of Theorem \ref{thm:h}.] 

Let $\tilde{R}_0$ as in Proposition \ref{prop:R} and take $R\geq \tilde{R}_0$.

Let $m:=  \int_{\R^N} \dot{u}(t,\cdot)p(t,\cdot)$ (an easy integration by parts yields that this quantity does not depend on $t$). We will prove the first case in the Theorem, that is, if $m\geq c\|h\|_{L^q(\R^N)}$, then $u^h(t,x)\to 1$ as $t\to +\infty$ locally in $c$. We first use Theorem \ref{thm:Huskaunbounded} to get:
$$\begin{array}{rcl}
\dot{u}(t,x)&=&mv(t,x)+ \dot{u}(t,x)-mv(t,x)\\
&&\\
&\geq & mv(t,x)-\| \dot{u}(t,\cdot)-mv(t,\cdot)\|_{L^\infty(\R^N)}\\
&&\\
&\geq & mv(t,x)-C e^{-\gamma t}\| \dot{u}(1,\cdot)-mv(1,\cdot)\|_{L^\infty(\R^N)}\displaystyle\frac{\|v(t,\cdot)\|_{L^\infty(\R^N)}}{\|v(1,\cdot)\|_{L^\infty(\R^N)}}\\
&&\\
&= &\|v(t,\cdot)\|_{L^\infty(\R^N)}\Big( m\displaystyle\frac{v(t,x)}{\|v(t,\cdot)\|_{L^\infty(\R^N)}}-C e^{-\gamma t}\displaystyle\frac{\| \dot{u}(1,\cdot)-mv(1,\cdot)\|_{L^\infty(\R^N)}}{\|v(1,\cdot)\|_{L^\infty(\R^N)}}\Big).\\
\end{array}$$
As $m=  \int_{\R^N} \dot{u}(t,\cdot)p(t,\cdot)$, one has 
$$\begin{array}{rcl}
\| \dot{u}(1,\cdot)-mv(1,\cdot)\|_{L^\infty(\R^N)}&\leq& \| \dot{u}(1,\cdot)\|_{L^\infty(\R^N)}\big(1+\int_{\R^N} p(1,\cdot) \|v(1,\cdot)\|_{L^\infty(\R^N)}\big)\\
&&\\
&\leq &2\| \dot{u}(1,\cdot)\|_{L^\infty(\R^N)}\int_{\R^N} p(1,\cdot) \|v(1,\cdot)\|_{L^\infty(\R^N)}\\
\end{array}
$$
since $\int_{\R^N} p(1,\cdot)v(1,\cdot)=1$ (since this quantity is independent of $t$ and $\int_{\R^N} p(0,\cdot)v(0,\cdot)=1$ by normalization, see Proposition \ref{prop:defp}). 
Moreover, 
$$|\dot{u}(1,x)|\leq \frac{e^K}{(4\pi )^{N/2}}\int_{\R^N} e^{-\frac{|x-y|^2}{4}}|h(y)|dy\leq C_q\|h\|_{L^q(\R^N)},$$
where $K=\|\partial_u f\|_\infty$ and $C_q>0$ is a constant only depending on $q$, $K$ and $N$. 

We thus conclude that 
$$\dot{u}(t,x)\geq \|v(t,\cdot)\|_{L^\infty(\R^N)}\Big( m\frac{v(t,x)}{\|v(t,\cdot)\|_{L^\infty(\R^N)}}-2C C_q e^{-\gamma t} \int_{\R^N} p(1,\cdot) \|h\|_{L^q(\R^N)}\Big)$$
and as $m\geq c\|h\|_{L^q(\R^N)}$, one gets 
$$\dot{u}(t,x)\geq \|v(t,\cdot)\|_{L^\infty(\R^N)}\|h\|_{L^q(\R^N)}\Big( c\frac{v(t,x)}{\|v(t,\cdot)\|_{L^\infty(\R^N)}}-2CC_q e^{-\gamma t} \int_{\R^N} p(1,\cdot) \Big).$$

Lemma \ref{lem:cvphi} yields that there exists $T>0$ such that for all $t\geq T$ and $x\in B_R$:
$$\frac{v(t,x)}{\|v(t,\cdot)\|_{L^\infty(\R^N)}}\geq \frac{1}{2}\min_{B_R} \varphi.$$
Moreover, we could assume that $T$ is large enough so that 
$$2C e^{-\gamma T} C_q\int_{\R^N} p(1,\cdot) \leq  \frac{c}{4}\min_{B_R} \varphi.$$
It follows that 
$$\min_{B_R} \dot{u}(T,\cdot)\geq \frac{c}{4}\|v(T,\cdot)\|_{L^\infty(\R^N)}\min_{B_R} \varphi \|h\|_{L^q(\R^N)}.$$

As $f\in \mathcal{C}^2([0,1])$, there exists $M>0$ such that for all $t>0$, $x\in \R^N$, one has
$$| f\big( u^h(t,x)\big)-f\big(u(t,x)\big)-\partial_u f\big(x,u(t,x)\big)(u^h(t,x)-u(t,x))||\leq M|u^h(t,x)-u(t,x)|^2.$$
Let $z^h:= u^h-u- \dot{u}$. One has $z^h (0,\cdot)\equiv 0$ and 
$$\partial_t z^h - \Delta z^h-\partial_u f\big( x,u(t,x)\big)z^h = f(x, u^h)-f(x,u)-\partial_u f(x,u)(u^h-u).$$
Hence, 
$$|\partial_t z^h - \Delta z^h|\leq M|u^h-u|^2+K|z^h|,$$
where $K=\|\partial_u f\|_\infty$. Moreover, 
$$| u^h(t,x)-u(t,x)|\leq \frac{e^{Kt}}{(4\pi t)^{N/2}}\int_{\R^N} e^{-\frac{|x-y|^2}{4t}}h(y)dy\leq e^{Kt} C_q \|h\|_{L^q (\R^N)}.$$
It follows from comparison arguments that 
$$\|z^h(t,\cdot)\|_{L^\infty (\R^N)}\leq C_q M t^{1-N/2} e^{Kt}\|h\|_{L^q (\R^N)}^2.$$
Take $\e_q$ small enough such that 
$$ C_q M T^{1-N/2} e^{KT}\e_q\leq \frac{c}{8}\|v(T,\cdot)\|_{L^\infty(\R^N)}\min_{B_R} \varphi.$$
This defines $\e_q$. 
Then as $\|h\|_{L^q (\R^N)}<\e_q$, one has 
$$\|z^h(T,\cdot)\|_{L^\infty (\R^N)}\leq \displaystyle\frac{c }{8}\|v(T,\cdot)\|_{L^\infty(\R^N)}\|h\|_{L^q (\R^N)}\min_{B_R} \varphi $$ and thus for all $x\in B_R$:
$$\begin{array}{l}
h^T(x):=u^h (T,x)-u(T,x)\\
\\
=\displaystyle z^h (T,x)+\dot{u}(T,x)\geq \Big(-\frac{c}{8}\|v(T,\cdot)\|_{L^\infty(\R^N)}\min_{B_R} \varphi+\frac{c}{4}\|v(T,\cdot)\|_{L^\infty(\R^N)}\min_{B_R} \varphi\Big) \|h\|_{L^q (\R^N)}\\
\\
 =\displaystyle \frac{c}{8}\|v(T,\cdot)\|_{L^\infty(\R^N)}\|h\|_{L^q (\R^N)}\min_{B_R} \varphi.\\
 \end{array}$$
 Moreover, 
 $$|h^T(x)|=|u^h (T,x)-u(T,x)|\leq e^{KT} C_q \|h\|_{L^q (\R^N)}.$$
 Hence, $\|h^T\|_{L^\infty (\R^N)}<\e_\infty$ if $\|h\|_{L^q (\R^N)}<\e_q := e^{-KT}\e_\infty/ C_q $. 

We could thus apply Proposition \ref{prop:R}, with initial time at $t=T$. This yields that if we denote by $u_T^h$ the solution of 
$$\partial_t u_T^h- \Delta u_T^h =f(x,u_T^h) \quad \hbox{ on } (T,\infty) \times \R^N, \quad u_T^h (T,x)=u(T,x)+h^T(x)  \hbox{ for all  } x\in \R^N$$
then $\inf_{B_{R-1}} \big(u_T^h (t,\cdot)-u_T^0 (t,\cdot)\big)>0$ for all $t>0$. But, due to the definition of $h^T$, $u_T^h (T,\cdot)\equiv u^h (T,\cdot)$ and thus $u_T^h$ is nothing else but the function $u^h$. Hence, $\inf_{B_{R-1}} \big(u^h (t,\cdot)-u (t,\cdot)\big)>0$ for all $t>T$ and thus $\liminf_{t\to +\infty}\big(u^h (t,x)-W(x)\big)>0$ for all $x\in B_{R-1}$. As this is true for any $R$ large enough, this concludes the proof. 
\end{proof}


\subsection{Proof of Corollary \ref{cor:h}}

\begin{proof}[Proof of Corollary \ref{cor:h}.]
As $0\leq u \leq 1$, parabolic regularity estimates yield $\partial_t u \in L^2 \big( (0,T), L^2 (\R^N)\big)$ for all $T>0$. 
The Lebesgue differentiation  theorem yields that for a.e. $t_0\in [0,T)$, one has:
$$\frac{1}{\tau}\|u(\tau+t_0, \cdot)-u(t_0,\cdot)\|_{L^2 (\R^N)}\to \|\partial_t u(t_0,\cdot)\|_{L^2(\R^N)} \quad \hbox{ as } \tau\to 0$$

We take such a $t_0$. 
Assume by contradiction that $\int_{\R^N} p(t_0,x)\partial_t u(t_0,x)dx>0$. Let $c>0$ such that 
$$\int_{\R^N} p(t_0,x)\partial_t u(t_0,x)dx= 2c \|\partial_t u (t_0,\cdot)\|_{L^2 (\R^N)}.$$ 
Let $\e_2$ associated with $c$ as in Theorem \ref{thm:h}, but with initial time $t_0$ instead of $0$. 

For all $\tau>0$, we define $h^\tau:= u(\tau+t_0, \cdot)-u(t_0,\cdot)$. This is clearly an admissible perturbation since $0\leq u(\tau+t_0,\cdot)\equiv u(t_0,\cdot)+h\leq 1$. 
Clearly, $u^{h^\tau}\equiv u(\cdot+\tau+t_0,\cdot)$ and in particular, $u^{h^\tau}(t,x)\to W(x)$ as $t\to +\infty$ uniformly in $x\in \R^N$. 

On the other hand, $\frac{1}{\tau}\|h^\tau\|_{L^2 (\R^N)}\to \|\partial_t u(t_0,\cdot)\|_{L^2(\R^N)}$  
and $\int_{\R^N} p(t_0,x)h^\tau (x)dx\to \int_{\R^N} p(t_0,x)\partial_t u(t_0,x)dx$ as $\tau\to 0$. 
Hence, we can assume that $\tau$ is small enough so that 
$$\int_{\R^N} p(t_0,x) h^\tau (x)dx\geq c \|h^\tau\|_{L^2 (\R^N)} \quad \hbox{ and }  \|h^\tau\|_{L^2 (\R^N)} <\e_2.$$ 
As $c>0$, it would then follow that $u^{h^\tau}(t,x)\to 1$ as $t\to +\infty$ locally in $x\in \R^N$, a contradiction. Similarly, if $\int_{\R^N} p(t_0,x)\partial_t u(t_0,x)dx<0$, then $u^{h^\tau}(t,x)\to 0$ as $t\to +\infty$, providing the contradiction. Hence, $\int_{\R^N} p(t_0,x)\partial_t u(t_0,x)dx=0$, and this is true for a.e. $t_0\geq 0$. 
\end{proof}


\section{Derivation of the applications results}\label{sec:appli}

\begin{proof}[Proof of Proposition \ref{prop:app1}.] Let $r>0$, $L=L^*(r)$ the associated critical length. 
First, by definition of $L^* (r)$, one has $u_{L^*(r)}(t,x)\to W(x)$ as $t\to +\infty$ uniformly in $x\in \R$ for some positive, even solution $W\in H^1 (\R)$ of (\ref{eq:groundstate}). Hence, the hypotheses of Theorem \ref{thm:h} are satisfied and $p$ is well-defined. 

For all $\e>0$ small and $\nu\in \R$, we define $u^\e$ the solution of (\ref{eq:princ}) associated with the initial datum 
$$u^\e_0:= \left\{ \begin{array}{ll}1 &\hbox{ if }  r+\e<|x|<r+\e+L+\nu \e,\\
0 &\hbox{ otherwise}.\\
\end{array}\right.$$ 
We now denote $u_0:= 1_{(-L -r,-r)\cup (r,r+L)}$ and $u$ the associated solution. 

Define
$$h^\e:= u^\e_0-u_0=  \left\{ \begin{array}{ll}1 &\hbox{ if } L+r<|x|<r+\e+L+\nu \e,\\
-1 &\hbox{ if } r<|x|<r+\e,\\
0 &\hbox{ otherwise}.\\
\end{array}\right.$$ 

Assume first that $\nu>\frac{p^r(0,r)}{p^r(0,L+r)}-1$. Then 
$$\lim_{\e\to 0}\frac{1}{\e}\int_\R h^\e(x) p(0,x)dx= 2\nu p^r (0,L+r)-2p^r(0,r)>0.$$ 
On the other hand, $\|h^\e\|_{L^1 (\R^N)}=2(1+\nu)\e$. Let $c:= \big(\nu p^r (0,L+r)-p(0,r)\big)/(1+\nu)>0$. We could apply Theorem \ref{thm:h}: there exists $\e_1>0$ such that for all $\e$ such that $2(1+\nu)\e<\e_1$, one has 
$u^\e(t,x)\to 1$ locally uniformly in $x$ as $t\to +\infty$. By definition of $L^*(r+\e)$, this implies in particular that $L^*(r+\e)<\nu \e +L^*(r)$ for $\e$ small enough for all $\nu>\frac{p^r(0,r)}{p^r(0,L+r)}-1$.
Similarly, one can show that $L^*(r+\e)>\nu \e +L^*(r)$ for $\e$ small enough for all $\nu<\frac{p^r(0,r)}{p^r(0,L+r)}-1$. Analogous inequalities could be derived for $\e<0$. 
We conclude that $L^*$ is derivable on $(0,\infty)$ and $(L^*)'(r)=\frac{p^r(0,r)}{p^r(0,L^*(r)+r)}-1$. 
\end{proof}

%
%
%
%


\begin{proof}[Proof of Lemma \ref{lem:app1}.]
Consider a minimizing sequence $(u_0^n)_n$ in $\mathcal{B}$. As $0\leq u_0^n\leq 1$ for all $n$ and $\int_\R j(u_0^n)$ is bounded, we can assume that $\big(j(u_0^n)\big)_n$ converges weakly in the space of measures. Let $\beta:=\lim_{n\to+\infty}\int_\R j(u_0^n)$.

{\bf Step 1.}
Assume that there exists $\alpha \in (0,\beta)$ such that for all $\e \in (0,\beta-\alpha)$, there exist 
two compactly supported sequences $(u_{0,1}^n)_n$ and $(u_{0,2}^n)_n$ such that for $n$ large enough:
$$\|j(u_0^n)-j(u_{0,1}^n)-j(u_{0,2}^n)\|_{L^1 (\R)}<\e, \quad \big|\int_\R j(u_{0,1}^n)- \alpha\big|<\e$$
and $\lim_{n\to+\infty} d\big( supp u_{0,1}^n, supp u_{0,2}^n\big)=\infty$. 
For all $n$, let $u_1^n$ and $u_2^n$ the solution of (\ref{eq:princ}) associated respectively with the initial data $u_{0,1}^n$ and $u_{0,2}^n$. 

This part of the proof is now inspired by the proof of Theorem 2 in \cite{Garnier}. Let $v^n$ the solution associated with the initial datum $v^n_0:=u^n_0-u_{0,2}^n$. 
As $\alpha <\beta$ and $\e<\beta-\alpha$, one has $\int_\R j(v^n_0)<\beta$ and we know by definition of $\beta$ that $\lim_{t\to \infty} \sup_{x\in \R}v^n(t,x)=0$. Let $t^*>0$ such that $v^n(t^*,x)\leq \theta/4$ for all $x\in \R$, where $\theta:= \inf \{ s\in (0,1], f(s)>0\}$, which is well-defined since $f(1)=0$, and positive since $f'(0)<0$. 

Next, up to symmetrization, we can always assume that $a_n:=\sup  supp u_{0,1}^n\leq 0$ and $b_n:=\inf  supp \ u_{0,2}^n\geq 0$, and by translation we can assume that $b_n=-a_n$. As $\lim_{n\to+\infty} d\big( supp \ u_{0,1}^n, supp \ u_{0,2}^n\big)=\infty$, one has $b_n\to +\infty$ since $d\big( supp \ u_{0,1}^n, supp \ u_{0,2}^n\big)=2b_n$. 

Let $w^n=u^n-v^n$. One has $w^n(0,\cdot)\equiv u_{0,2}^n$ and 
$$w^n_t-w^n_{xx}=f(u^n)-f(v^n)\leq Kw^n,$$
where $K=\|f'\|_\infty$. 
Moreover, $supp w^n \subset (b_n,\infty)$, and thus $w^n (0,\cdot)$ converges weakly to $0$ as $n\to +\infty$ since $b_n\to +\infty$. It follows that the function $w^n (t,\cdot)$ converges locally uniformly to $0$ as $n\to +\infty$ for all $t>0$. In particular, 
$w^n(t^*,0)\leq \theta/4$ for $n$ large enough. Moreover, for all $t>0$, $w^n(t,\cdot)$ is increasing over $\R^-$ by Lemma 2.1 of \cite{DuMatano}, and thus $w^n(t^*,x)\leq \theta/4$ for all $x\leq 0$. It follows that $u^n(t^*,x)=w^n(t^*,x)+v^n(t^*,x)\leq \theta/2$ for all $x\leq 0$. One can prove using $u_1^n$ instead of $u_2^n$ that $u^n(t^*,x)\leq \theta/2$ for all $x\geq 0$. Hence $u^n(t^*,\cdot )\leq \theta/2$ over $\R$.

Consider the solution $\mathcal{N}$ of $\mathcal{N}'=f(\mathcal{N})$, $\mathcal{N}(0)=\theta/2$. As $f(s)<0$ for all $s\in (0,\theta)$, one has $\mathcal{N}(t)\to 0$ as $t\to +\infty$. Thus, as $u^n(t,\cdot)\leq \mathcal{N}(t)$ for all $t\geq t^*$, we conclude that $u^n(t,x)\to 0$ as $t\to +\infty$ uniformly over $\R$, for any $n$ large enough. 
This is a contradiction. Hence, dichotomy is discarded.

\smallskip

{\bf Step 2.}
Assume that  $(u_0^n)_n$ vanishes, in the sense that for all $R>0$:
$$\sup_{y\in \R} \int_{B_R+y}j( u_0^n) \to 0 \quad \hbox{ as } n\to +\infty.$$

Let $K=\|f'\|_\infty$. Let $\e>0$ such that $j(s)>\e$ for all $s>2\theta e^{-K} /3$. Take $R$ large enough so that 
$$\frac{1}{\sqrt{4\pi}}\int_{|x|>R}e^{-|x|^2/4}dx<\e/(2j(1)).$$
Assume that $n$ is large enough so that 
$$\frac{1}{\sqrt{4\pi}}\sup_{y\in \R} \int_{B_R+y} j(u_0^n)  \leq \e/2.$$
We now notice that 
$$u^n_t-u^n_{xx}=f(u^n)\leq Ku^n \hbox{ in } (0,\infty)\times \R.$$
It follows that, as $u_0^n\leq 1$, using Jensen's inequality and the convexity of $j$:
$$\begin{array}{rcl} 
 j\big( e^{-K}u^n(1,y)\big)&\leq&  j\Big(\frac{1}{\sqrt{4\pi }}\int_\R e^{-|y-z|^2/4}u_0^n (z)dz\Big)\\
&&\\
&\leq&  \frac{1}{\sqrt{4\pi }}\int_\R e^{-|y-z|^2/4}j\big(u_0^n (z)\big) dz\\
&&\\
&\leq &\frac{1}{\sqrt{4\pi }}\int_{|z-y|\leq R} j\big(u_0^n (z)\big)dz+\frac{1}{\sqrt{4\pi }}\int_{|z-y|>R} e^{-|y-z|^2/4}j(1)dz\\
&&\\
&\leq &\e\\ \end{array}$$
for $n$ large enough. It follows from the definition of $\e$ that $e^{-K}u^n(1,y)\leq 2\theta e^{-K} /3$. 
Comparing with the solution $\mathcal{N}$ of $\mathcal{N}'=f(\mathcal{N})$, $\mathcal{N}(0)=2\theta/3$, we conclude that $u^n(t,y)\to 0$ as $t\to +\infty$ for $n$ large enough uniformly in $y\in \R$. This is a contradiction since $u_0^n\in \mathcal{B}$. Hence, vanishing is discarded.

\smallskip

{\bf Step 3.} We conclude from Lions' concentration-compactness method \cite{Lions} that, up to translation, for all $\e>0$, there exists $R>0$ such that $\int_{B_R} j(u^n_0)\geq \beta - \e$ for all $n$ large enough. 
As $0\leq j(u_0^n)\leq j(1)$ for all $n$, it follows that $\big( j(u_0^n)\big)_n$ converges, up to extraction, in $L^1$ (see \cite{Lions}). 
As $j$ is bijective on $[0,\infty)$, let $j(\underline{u}_0)$ its limit.

It is only left to prove that if $\underline{u}$ is the solution of (\ref{eq:princ}) associated with the initial datum $\underline{u}_0$, then $\limsup_{t\to+\infty} \sup_{x\in \R}\underline{u}(t,x)>0$. Assume by contradiction that $\underline{u}(t,\cdot)$ converges uniformly to $0$ as $t\to+\infty$. Let $t^*$ such that $\underline{u}(t^*,x)\leq \theta/4$ for all $x\in \R$. 
 Let $\e>0$ such that $j(s)> \frac{\e}{\sqrt{4\pi t^*}}$ for all $s>\theta e^{-Kt^*} /4$. Take $n$ large enough so that $\|j(u_0^n)-j(\underline{u}_0)\|_{L^1 (\R)}\leq \e$. We notice that, as $j$ is convex, nondecreasing and $j(0)=0$, one has $j(u+v)\geq j(u)+j(v)$ for all $u,v\geq 0$, from which we could easily derive that $j(|u-v|)\leq |j(u)-j(v)|$ for all $u,v\geq 0$. It follows that 
$$\begin{array}{rcl} 
j\Big(e^{-K t^*} |u^n(t^*,x)-\underline{u}(t^*,x)|\Big)&\leq & j\Big( \frac{1}{\sqrt{4\pi t^*}}\int_\R e^{-\frac{|x-y|^2}{4t^*}}|u^n_0(y)-\underline{u}_0(y)|dy \Big)\\
&&\\
&\leq & \frac{1}{\sqrt{4\pi t^*}}\int_\R e^{-\frac{|x-y|^2}{4t^*}}j\big( |u^n_0(y)-\underline{u}_0(y)|\big) dy \quad \hbox{(Jensen inequality)}\\
&&\\
&\leq & \frac{1}{\sqrt{4\pi t^*}}\int_\R e^{-\frac{|x-y|^2}{4t^*}}|j\big( u^n_0(y)\big)-j\big(\underline{u}_0(y)\big)| dy \\
&&\\
&\leq &  \frac{\e}{\sqrt{4\pi t^*}}.\\
\end{array}$$
It follows from the definition of $\e$ that $|u^n(t^*,x)-\underline{u}(t^*,x)|\leq \theta/4$ and thus $u^n(t^*,x)\leq \theta/2$ for all $x\in \R$. By comparison with $\mathcal{N}$, one gets $\lim_{t\to +\infty} u^n(t,x)=0$ uniformly in $x\in \R$. This contradicts $u^n_0\in \mathcal{B}$. We conclude that $\underline{u}_0\in \mathcal{B}$. 
\end{proof}

\begin{proof}[Proof of Proposition \ref{prop:bathtub}.]
First, by (\ref{hyp:f}) and (\ref{hyp:fbist1}), as $\underline{u}_0\in L^1 (\R)$, it follows from  \cite{MatanoPolacik} applied to $\underline{u}(1,\cdot)$ that $\underline{u}(t,\cdot)$ converges as $t\to +\infty$, either locally uniformly to $1$,  uniformly to $0$, or uniformly to a positive solution $W\in H^1 (\R)$ of (\ref{eq:groundstate}). The limit $0$ is excluded since $\underline{u}_0 \in \mathcal{B}$. 

Assume by contradiction that $\underline{u}(t,\cdot)\to 1$  as $t\to +\infty$, locally uniformly in $x$. Then it is well-known (see Lemma 4.2 of \cite{DuMatano} for example) that under hypotheses (\ref{hyp:f}) and (\ref{hyp:fbist1}), for any $\alpha \in (\beta^*,1)$, there exists $R_\alpha$ such that the solution $u^\alpha$ associated with the initial datum $\alpha \mathbf{I}_{(-R_\alpha, R_\alpha)}$ converges to $1$ as $t\to +\infty$ locally uniformly in $x$. Take $T>0$ such that  $\underline{u}(t,x)\geq (1+\alpha)/2$  for all $t>T$ and $x\in (-R_\alpha, R_\alpha)$. Let $h\in L^1 (\R)$ a nonnegative function such that $\underline{u}_0\geq h$, and $\frac{e^{KT}}{\sqrt{4\pi T}}\|h\|_{L^1 (\R)}<(1-\alpha)/2$, with $K=\|f'\|_\infty$. Let $u^h$ the solution associated with $\underline{u}_0- h$. Using arguments that have already been developed in this paper, one easily gets 
$$\|u^h(T,\cdot)-\underline{u}(T,\cdot)\|_{L^\infty(\R)}\leq \frac{e^{KT}}{\sqrt{4\pi T}}\|h\|_{L^1 (\R)}<(1-\alpha)/2.$$
In particular, $u^h(T,x)\geq \alpha $ on $(-R_\alpha,R_\alpha)$ and thus $u^h(t,x)\to 1$ as $t\to +\infty$ locally in $x$. But as $\int_\R j\big(u^h(0,\cdot)\big)<\int_\R j(\underline{u}_0)$, this would contradict the minimality of $\underline{u}_0$.

Hence the only possible choice is $\underline{u}(t,x)\to W(x)$ uniformly in $x$ as $t\to +\infty$ and thus the hypotheses of 
Theorem \ref{thm:h} are satisfied. We can thus define $p$. 

Define the set $\mathcal{S}^0=\{x\in \R, \overline{u}_0(x)<1\}$ (resp. $\mathcal{T}^0=\{x\in \R, \overline{u}_0(x)>0\}$) and for every $k=1,2,\dots $ the set $\mathcal{S}^0_k=\{x\in \R, 0\leq \overline{u}_0(x)<1-\frac{1}{k}\}$ (resp.  $\mathcal{T}^0_k=\{y\in \R, 0\leq \overline{u}_0(y)>\frac{1}{k}\}$); 
As $\overline{u}_0\in L^1 (\R)$, one has $\overline{u}_0\not\equiv 1$ and thus $meas \mathcal{S}^0>0$. Also, 
$\overline{u}_0\not\equiv 0$, otherwise one would have $u(t,\cdot)\equiv 0$ for all $t>0$, contradicting $ \limsup_{t\to+\infty} \sup_{x\in \R}u(t,x)>0$, and thus $meas \mathcal{T}^0>0$

The Lebesgue density theorem yields that for almost every $x^* \in \mathcal{S}^0_k,  y^* \in \mathcal{T}^0_k$, there exists $r$ sufficiently small such that ${\frac{\mu(B(x^*,r))}{\mu(B(x^*,r)\cap \mathcal{S}_k^0)}}<2$ and ${\frac{\mu(B(y^*,r))}{\mu(B(y^*,r)\cap \mathcal{T}_k)}}<2$. Take $\lambda>1$, and define
	    \begin{equation*}
		    h(x):= \frac{meas(B(x^*,r))}{meas(B(x^*,r)\cap \mathcal{S}_k^0)} \frac{\mathbf{I}_{B(x^*,r)\cap \mathcal{S}_k^0}}{j'(\underline{u}_0)}-\lambda\frac{meas(B(y^*,r))}{meas(B(y^*,r)\cap \mathcal{T}_k)} \frac{\mathbf{I}_{B(y^*,r)\cap \mathcal{T}_k}}{j'(\underline{u}_0)},
	    \end{equation*}
	    where $\mathbf{I}$ is the indicatrix function. 
Then for $0<\e<\frac{1}{2k}$, one has $0\leq \underline{u}_0+\e h\leq 1$. As $\lambda>1$, one has $\int_\R j'(\underline{u}_0) h<0$ and thus $\int_\R j(\underline{u}_0+\e h)<\int_\R j(\underline{u}_0)$. The minimality of $\underline{u}_0$ yields that $\underline{u}_0+h\not\in \mathcal{A}$, that is, 
the solution $u^h$ associated with the initial datum $\underline{u}_0+\e h$ converges uniformly to $0$ as $t\to +\infty$. 

Assume by contradiction that $\int_\R p(0,x)h(x)>0$. Then Theorem \ref{thm:h} yields that $u^{\e h}(t,x)\to 1$ locally uniformly in $x$ as $t\to +\infty$, a contradiction. Hence $\int_\R p(0,x)h(x)<0$, that is:
$$\frac{1}{meas(B(x^*,r)\cap \mathcal{S}_k^0)} \int_{B(x^*,r)\cap \mathcal{S}_k^0} \frac{p(0,x)}{j'(\underline{u}_0 (x))}dx \leq \lambda\frac{1}{meas(B(y^*,r)\cap \mathcal{T}_k)} \int_{B(y^*,r)\cap \mathcal{T}_k} \frac{p(0,x)}{j'(\underline{u}_0 (x))}dx.$$
Letting $r\to 0$, this yields $ \frac{p(0,x)}{j'(\underline{u}_0 (x))}\leq \lambda  \frac{p(0,y)}{j'(\underline{u}_0 (y))}$ for all $\lambda>1$. As this is true for almost every $x\in \mathcal{S}^0_k$ and $y\in \mathcal{T}^0_k$, and as $\mathcal{S}^0=\cap_{k=1}^{\infty}\mathcal{S}^0_k$ and  $\mathcal{T}^0=\cap_{k=1}^{\infty}\mathcal{T}^0_k$,we have thus proved that 
$$  \frac{p(0,x)}{j'(\underline{u}_0 (x))}\leq   \frac{p(0,y)}{j'(\underline{u}_0 (y))}  \hbox{ for almost every } x\in \mathcal{S}^0, y\in \mathcal{T}^0.$$
As $p(0,\cdot)$ is continuous, the conclusion follows by letting $c:=\sup_{\mathcal{S}^0} \frac{p(0,\cdot)}{j'(\underline{u}_0)}=\inf_{\mathcal{T}^0} \frac{p(0,\cdot)}{j'(\underline{u}_0)}$ since $\mathcal{S}^0\cup \mathcal{T}^0= \R^N$. 
\end{proof}

\begin{proof}[Proof of Corollary \ref{cor:app1}.]
Assume by contradiction that  $W=W(x)$ is a minimizer of $u_0\mapsto \int_\R u_0$ over $\mathcal{B}$. If $u_0=W$, then, clearly, $u(t,x)=W(x)$ for all $(t,x)\in (0,\infty)\times \R$, and $p(t,x)=C\varphi (x)e^{\lambda t}$ for some constant $C>0$ by uniqueness. As $0<W<1$ over $\R$, Proposition \ref{prop:bathtub} would give $p(0,\cdot)\equiv c$ over $\R$ for some $c>0$. In other words, $\varphi$ would be constant, and thus $f'(W)$ would be constant over $\R$, equal to $f'(0)<0$ since $W(x)\to 0$ as $|x|\to +\infty$. This would give $f(s)=f'(0)s$ for all $s\in [0,\max_\R W]$. On the other hand, equation (\ref{eq:groundstate}) yields that if $W(\overline{x})=\max_\R W$, then 
$f\big(W(\overline{x}\big)=-\Delta W (\overline{x})\geq 0$, contradicting $f'(0)<0$. 
\end{proof}



\bibliographystyle{plain}

\bibliography{biblio}


\end{document}